\documentclass[12pt]{article}
\usepackage{amssymb}
\usepackage{graphicx}
\usepackage{amsfonts}
\usepackage{latexsym}
\usepackage{amsthm}
\usepackage{amsmath}

\usepackage{amssymb, amscd}

\usepackage{amsmath,amsfonts,graphicx,subfigure,url,amscd}

\usepackage{url}
\usepackage[T1]{fontenc}
\usepackage{color}
\usepackage{tabu}

\numberwithin{equation}{section}

\newtheorem{theo}[equation]{Theorem}
\newtheorem{rem}[equation]{Remark}

\newtheorem{defin}[equation]{Definition}
\newtheorem{prop}[equation]{Proposition}

\newtheorem{lema}[equation]{Lemma}

\begin{document}
\date{October 1, 2019}
\title{{Polytopal Bier spheres and Kantorovich-Rubinstein polytopes of weighted cycles}}

\author{{Filip D. Jevti\'{c}} \\{\small Mathematical Institute}\\[-2mm] {\small SASA,   Belgrade}
\and Marinko Timotijevi\'{c}\\ {\small Faculty of Science}\\[-2mm] {\small University of Kragujevac}
\and Rade T. \v Zivaljevi\' c\\ {\small Mathematical Institute}\\[-2mm] {\small SASA,   Belgrade}}

\maketitle
\begin{abstract}\noindent
The problem of deciding if a given triangulation of a sphere can be realized as the boundary sphere
of a simplicial, convex polytope is known as the `Simplicial Steinitz problem'. It is known by an indirect and non-constructive
argument that a vast majority of Bier spheres are non-polytopal. Contrary to that,
we demonstrate that the Bier spheres associated to threshold simplicial complexes are all polytopal. Moreover, we
show that all Bier spheres are starshaped. We also establish a connection between Bier spheres and
Kantorovich-Rubinstein polytopes by showing that the boundary sphere of the KR-polytope associated to a polygonal linkage
(weighted cycle) is isomorphic to the Bier sphere of the associated simplicial complex of {``short sets''}.
\end{abstract}



\renewcommand{\thefootnote}{}
\footnotetext{This research was supported by the Grants 174020 and
174034 of the Ministry of Education, Science and Technological Development of the Republic of Serbia.}

\section{Introduction}
\label{sec:intro}

The classic theory of the optimal transportation, as developed by L.~Kantorovich \cite{k42, Vil, Vil2}, is one of the pillars
of the theory of linear programming \cite{v13, Vil}. The central paradigm of the theory is the  Kantorovich duality principle \cite{Vil}, in its manifold forms and incarnations.
It includes, as one of the central consequences, the Kantorovich-Rubinstein theorem \cite[Theorem 1.14]{Vil}, which pertains to the case when the cost function is a metric.

\medskip
 Much more recent is the research program, proposed by A.~Vershik in \cite{v15}, of studying ``fundamental polytopes'' or Kantorovich-Rubinstein polytopes
 as a tool for classifying metric spaces from the view point of polyhedral combinatorics (see Section~\ref{sec:K-R}
for an outline).  These ideas can be traced back to Vershik's earlier publications
\cite{v13, v14}, especially \cite{mpv} (with J.~Melleray and F.~Petrov) and to Kantorovich himself, see \cite{kr} where the Kantorovich-Rubinstein norm $\|\mu - \nu  \|_{KR}$ is introduced.


\medskip
Bier spheres $Bier(K)$, where $K\subsetneq 2^{[n]}$ is an abstract simplicial complex, are combinatorially defined triangulations of
the $(n-2)$-dimensional sphere $S^{n-2}$ with interesting combinatorial and topological properties, \cite{long-1, M}. These spheres are known to be shellable \cite{bpsz, cukic}.
Moreover it is known, by an indirect and non-effective counting argument, that the majority of these spheres are non-polytopal,
in the sense that they do not admit a convex polytope realization, see \cite[Section 5.6]{M}. They also provide one of the most elegant proofs of the Van Kampen-Flores theorem \cite{M}
and serve as one of the main examples of ``Alexander complexes'' \cite{jnpz}.

\medskip
Threshold complexes are ubiquitous in mathematics and arise, often in disguise and under different names,   in areas as different as cooperative game theory (quota complexes and simple games) and algebraic topology
of configuration spaces (polygonal linkages, complexes of short sets)   \cite{CoDe, far, ga-pa}.

\medskip
The following theorem  establishes a connection between the boundary $\partial KR(d_L)$ of the Kantorovich-Rubinstein polytope of
a weighted cycle, and the Bier sphere of the threshold complex of ``short sets'' of the associated polygonal linkage.

\begin{theo}\label{thm:KR-Bier}
Suppose that $L = (l_1,l_2,\dots, l_n)\in \mathbb{R}^n_+$
is a strictly positive vector such that $\sum_{i=1}^n~l_i =1$. Let  $\mu_L$ be the associated measure (weight distribution) on $[n]$
and  let ${\rm Short}(L) = T_{\mu_L < 1/2} := \{I\subseteq [n] \mid \mu_L(I) < \frac{1}{2}\}$ be the associated simplicial complex of ``short sets''. We assume that
$L$ is generic in the sense that $(\forall I\subset [n])~\mu_L(I)\neq \frac{1}{2}$. Let $\hat{L}$
be a weighted cycle (linkage)  with bars of the length $l_i$ and let $d_L$ be the associated geodesic distance function on $[n]$. Then
\begin{equation}\label{eqn:geodesic}
\partial KR(d_L) \cong Bier({\rm Short}(L))  \, .
\end{equation}
\end{theo}

The proof of Theorem~\ref{thm:KR-Bier}, with all preliminary definitions and introductory facts, can be found in Section~\ref{sec:K-R}.

\medskip
As a corollary of (\ref{eqn:geodesic}) we observe that the Bier spheres associated to complexes of generic short sets are always
polytopal. In Section~\ref{sec:polytopal} (Theorem~\ref{thm:Bier-polytope}) we prove a more general result, somewhat surprising and interesting in itself, that the Bier spheres
of any threshold complex (with an arbitrary quota and not necessarily generic) is polytopal. This should be compared to the fact (proven by an indirect and non-constructive argument) that a vast majority
of Bier spheres are non-polytopal.

\medskip
By an old result of Steinitz all triangulations of $S^2$ are polytopal and the problem of testing if a triangulation of a sphere is polytopal
is known as the ``Simplicial Steinitz problem''. A closely related problem is the study of the asymptotic behavior of the number of nonisomorphic combinatorial types
of triangulated (shellable, polytopal, starshaped, etc.) spheres with $n$-vertices.  Early work of Goodman and Pollack \cite{GP-1, GP-2}, together with the estimates of Kalai \cite{Ka}, showed
that asymptotically very few triangulated spheres are polytopal. Far reaching new results of this type, as well as a guide to some of the more recent publications, can be found in \cite{B-Z, NSW, P-Z}.

\medskip
 Recall that not all triangulations of $(n-2)$-dimensional spheres are starshaped in the sense that they admit a starshaped geometric realization in $\mathbb{R}^{n-1}$.
 An example of such a sphere can be found in \cite[Theorem~5.5]{Ewald}. Our third main result, Theorem~\ref{thm:Bier-starshaped}, claims that all Bier spheres
 (associated to all simplicial complexes $K\subsetneq 2^{[n]}$) are starshaped.

 \medskip
 The notation and terminology in the paper is fairly standard. The book \cite{Ewald} is a general reference for the geometry of convex sets while the book \cite{M} provides an interesting
 and gentle introduction to combinatorial topology, with the emphasis on applications to combinatorics and discrete geometry.

\section{Polytopal Bier spheres}
\label{sec:polytopal}

The Alexander dual of a simplicial complex $K\subsetneq 2^{[n]}$ is the complex $K^\circ = \{I^c \mid I\notin K\}$.
Suppose that $u_1+ u_2+\dots+ u_n = 0$ is a `minimal circuit' in $\mathbb{R}^{n-1}$, meaning that each proper subset of
the collection of vectors $\{u_i\}_{i=1}^n$ is linearly independent. Suppose that $L = (l_1,l_2,\dots, l_n)\in \mathbb{R}^n_+$
is a strictly positive vector. The associated measure (weight distribution) $\mu_L$ on $[n]$ is defined by
$\mu_L(I) = \sum_{i\in I}~l_i$ (for each $I\subseteq [n]$).

\medskip
Given a threshold $\nu>0$, the associated threshold complex is $T_{\mu_L < \nu} := \{I\subseteq [n] \mid \mu_L(I)< \nu\}$.
Without loss of generality we assume that $\mu_L([n]) = l_1+\dots+ l_n = 1$.

\begin{rem}{\rm  If $K = T_{\mu_L < \nu}$ is a threshold complex then $K = T_{\mu_L < \nu - \epsilon}$ for each sufficiently small
$\epsilon > 0$.  It follows that we may assume, without loss of generality, that $\mu_L(I)\neq \nu$ for each $I\subseteq [n]$
and, as a consequence, we may assume that the Alexander dual of $K$ is $K^\circ = T_{\mu_L \leq 1- \nu} = T_{\mu_L < 1- \nu}$. }
\end{rem}

For a simplicial complex $K\subset 2^{[n]}$ let $K^\circ$ its Alexander dual, and let $\Delta_S = {\rm Conv}\{e_i\}_{i\in S}$ be
the geometric simplex spanned by $S\subseteq [n]$.  Recall that for $K, L\subseteq 2^{[n]}$, the deleted join $K\ast_\Delta L$  is a subcomplex of the join
$\Delta_{[n]} \ast \Delta_{[n]}$ defined by
$K\ast_\Delta L := \{A\uplus B \mid A\in K, B\in L, A\cap B=\emptyset\}$.

\medskip
For $K\subsetneq 2^{[n]}$, the associated Bier sphere is the deleted join,
\begin{equation}\label{eq:Bier-def}
Bier(K):= K \ast_\Delta K^\circ \subset \Delta_{[n]} \ast_\Delta \Delta_{[n]} \cong \partial\lozenge_{[n]}
\end{equation}
where $\partial\lozenge_{[n]}$ is the boundary sphere of the $n$-dimensional cross-polytope $\lozenge_{[n]} = {\rm Conv}\{\pm e_i\}_{i=1}^n$.

\begin{theo}\label{thm:Bier-polytope}
$Bier(T_{\mu_L<\nu})$ is isomorphic to the boundary sphere of a convex polytope.
\end{theo}

\medskip\noindent
{\bf Proof:}  Let  $y_i = \frac{u_i}{l_i}$  which implies that $l_1y_1+l_2y_2+\dots+l_ny_n=0$ is (up to a multiplicative constant)
the unique linear dependence of these vectors.  Let $\alpha > 0$ a positive constant and $\beta = \frac{1}{\alpha}$.

 \medskip
 Let $\Delta := {\rm Conv}\{y_i\}_{i=1}^n \subset \mathbb{R}^{n-1}$ be the simplex spanned by $y_i$ and $\nabla_\alpha := -\alpha\Delta$
 the simplex spanned by the vectors $-\alpha y_i$. We want to show that there exists $\alpha>0$ such that the Bier sphere $Bier(T_{\mu_L<\nu})$
 is isomorphic to the boundary sphere of the convex polytope,
 \begin{equation}\label{eqn:conv-poly}
 Q_\alpha :=  {\rm Conv}(\Delta\cup\nabla_\alpha) = {\rm Conv}\{y_1,y_2,\dots, y_n, -\alpha y_1, -\alpha y_2, \dots, -\alpha y_n\}\, .
 \end{equation}

A linear transform of the collection of vectors (\ref{eqn:conv-poly}), representing vertices of the polytope $Q_\alpha$, is easily found
and can be read off from the following matrix relation,

\begin{equation}\label{eqn:Gale-matrix-1}
  \begin{matrix}
   [ y_1 & y_2 & \dots & y_n & -\alpha y_1 & -\alpha y_2 & \dots & -\alpha y_n]\\
   {}
  \end{matrix}
  \begin{bmatrix}
    L^T & \alpha I_n \\
    0 & I_n
  \end{bmatrix}
  =0
\end{equation}
where $L^T = (l_1,\dots, l_n)^T$ is a column vector and $I_n$ the identity $(n\times n)$-matrix.
If $y$ is the row matrix $y = [y_1 \, y_2 \, \dots \, y_n]$ then the relation (\ref{eqn:Gale-matrix-1}) can be rewritten as,

\begin{equation}\label{eqn:Gale-matrix-2}
  \begin{matrix}
   [ y & -\alpha y ]\\
   {}
  \end{matrix}
  \begin{bmatrix}
    L^T & \alpha I_n \\
    0 & I_n
  \end{bmatrix}
  =0 \, .
\end{equation}

Let  $z : \mathbb{R}^{n-1} \rightarrow  \mathbb{R}$ be a non-zero linear form such that the associated hyperplane
$H_z := \{x\in \mathbb{R}^{n-1} \mid \langle z, x \rangle = 1\}$ is a supporting hyperplane of $Q_\alpha$. The
corresponding face of the polytope $Q_\alpha$ is described by a pair $(I, J)$ of subsets of $[n]$ recording which
vertices of the polytope $Q_\alpha$ belong to the hyperplane $H_z$. More explicitly
\begin{equation}\label{eqn:hyper}
           Q_\alpha \cap H_z = {\rm Conv}(\{y_i\}_{i\in I}\cup\{-\alpha y_j\}_{j\in J}) \, .
\end{equation}
 It follows from (\ref{eqn:Gale-matrix-2}) that,

\begin{equation}\label{eqn:Gale-matrix-3}
  \begin{matrix}
   [ \langle z, y \rangle & \langle z, -\alpha y \rangle  ]\\
   {}
  \end{matrix}
  \begin{bmatrix}
    L^T & \alpha I_n \\
    0 & I_n
  \end{bmatrix}
  =0
\end{equation}
where $\langle z,y \rangle = [\langle z,y_1 \rangle \dots \langle z,y_n \rangle   ]$. It follows from (\ref{eqn:hyper})
that the ordered pair $(I,J)$ of subsets of $[n]$ must satisfy the following:

\begin{align}
(\forall i\in I) \, \langle z, y_i \rangle = 1   &&    (\forall j\in J) \, \langle z, -\alpha y_j \rangle = 1 \label{eqn:prva} \\
(\forall k\notin I) \, \langle z, y_k \rangle < 1   &&    (\forall k\notin J) \, \langle z, -\alpha y_k \rangle < 1 \label{eqn:druga}
\end{align}
From these relations it follows:

\begin{equation}\label{eqn:face-1}
I\cap J = \emptyset  \quad \mbox{and} \quad \emptyset \neq I\cup J \neq [n]
\end{equation}

\begin{equation}\label{eqn:face-3}
(\forall k\notin I\cup J) \, -\beta < \langle z, y_k \rangle < 1
\end{equation}
From (\ref{eqn:face-3}) and the relation,
\begin{equation}\label{eqn:face-4}
\sum_{i\in [n]} l_i\langle z, y_i \rangle = 0 = \sum_{i\in I} l_i - \sum_{j\in J} \beta l_j + \sum_{k\notin I\cup J} l_k \langle z, y_k \rangle
\end{equation}
we deduce the following inequalities,
\begin{equation}\label{eqn:face-5}
-\mu_L((I\cup J)^c)  < \mu_L(I) - \beta \mu_L(J) < \beta \mu_L((I\cup J)^c)  \, .
\end{equation}
The inequalities (\ref{eqn:face-5}) can be rewritten as follows,
\begin{equation}\label{face-6}
 \mu_L(J^c)  > \beta \mu_L(J) \quad \mbox{and} \quad \mu_L(I)  < \beta \mu_L(I^c) \, .
\end{equation}
In light of the assumption $\sum_{i=1}^n l_i =1$ we finally obtain the inequalities,
\begin{equation}\label{eqn:final}
\mu_L(I) < \frac{\beta}{1+\beta} \quad   \mbox{and} \quad \mu_L(J) < \frac{1}{1+\beta} \, .
\end{equation}

In other words each face of the polytope $Q_\alpha$, described by the equation (\ref{eqn:hyper}), is associated a simplex $(I,J)\in Bier(T_{\mu_L< \nu})$
where $\nu = \frac{\beta}{1+\beta}$.

\medskip Conversely, let $(I,J)\in Bier(T_{\mu_L< \nu})$ be a face of the Bier sphere. Then the equation (\ref{eqn:final}) is translated
back to (\ref{eqn:face-5}) and we can choose $\langle z, y_k \rangle$ (for $k\notin I\cup J$) satisfying (\ref{eqn:face-3}) such that the equality
(\ref{eqn:face-4}) is also satisfied.

\medskip
More explicitly, let $X = 1 -\mu_L(I) - \mu_L(J) = \sum_{i\notin I\cup J}~l_i = \mu((I\cup J)^c)$. Then
\begin{equation}\label{eq:XX}
 - X < \mu_L(I) - \beta\mu_L(J) < \beta X
\end{equation}
 and there exists $\gamma\in (0,1)$ such that
 \begin{equation}\label{eq:XXX}
  \mu_L(I) - \beta\mu_L(J) = -\gamma X + (1-\gamma)\beta X
\end{equation}
which is equivalent to
\begin{equation}\label{eq:XXXX}
  0 = \mu_L(I) - \beta\mu_L(J) + X(\gamma + (1-\gamma)(-\beta)) \, .
\end{equation}
 Let us choose $z$ such that $\langle z, y_k\rangle$ satisfies the equations (\ref{eqn:prva}) for $k\in I\cup J$ while for $k\notin I\cup J$
 \begin{equation}\label{eq:XXXXX}
 \langle z, y_k\rangle = \gamma\cdot 1 + (1-\gamma)(-\beta) \, .
 \end{equation}

This is possible in light of the equality (\ref{eqn:face-4}).
In turn this proves the validity of relations (\ref{eqn:prva}) and (\ref{eqn:druga}) and eventually leads to (\ref{eqn:hyper}).
This observation completes the proof of the theorem. \hfill $\square$

\section{Starshaped Bier spheres}
\label{sec:starshaped}

A $d$-dimensional triangulated sphere $\Sigma^{d}$ is {\em starshaped} if there exists an embedding $e: \Sigma^d\rightarrow \mathbb{R}^{d+1}$,
linear (affine) on simplices of $\Sigma^d$, and a point $c\in \mathbb{R}^{d+1}\setminus e(\Sigma^d)$, such that $[c, x] \cap [c,y] = \{c\}$ for each pair $x\neq y$ of
distinct points in $e(\Sigma^d)$.

\medskip
As shown by Ewald and Schulz in \cite{e-s}, for each $d\geq 4$ there exists a $(d-1)$-dimensional simplicial sphere which cannot be embedded in
$\mathbb{R}^d$ as a starshaped set. An example of such a sphere is also described in \cite[Theorem~5.5]{Ewald}.

\medskip
Surprisingly enough all Bier spheres turn out to be starshaped. As a consequence Bier spheres provide (at least statistically)
numerous examples of non-polytopal, starshaped spheres. Indeed, according to \cite{M} there are more than $2^{(2^n/n)-2n^2}$ nonisomorphic Bier spheres, while the number of different combinatorial types of $(n-1)$-dimensional, simplicial convex polytopes with $2n$ vertices is not larger than $2^{4n^3}$.

Note that an exponential upper bound to the number of starshaped sets in terms of the number of facets was recently
proven in \cite[Theorem 2.5]{AB}.

\medskip
From here on we make a clear distinction between combinatorial, geometric, and
topological (deleted) join of simplicial complexes, as emphasized and discussed in \cite[Section~4.2]{M}.
For example the ``combinatorial deleted join'' representation
$\Delta_{[n]} \ast_\Delta \Delta_{[n]} \cong \partial\lozenge_{[n]}$, used in (\ref{eq:Bier-def}),
naturally leads to a ``geometric deleted join'' representation
\begin{equation}\label{eq:join-geometric}
\Delta_{e} \ast_\Delta \Delta_{-e} = \partial\lozenge_{[n]}
\end{equation}
where $\Delta_{e} = {\rm Conv}\{e_i\}_{i=1}^n$ and  $\Delta_{-e} = {\rm Conv}\{-e_i\}_{i=1}^n$. More generally, for each (labeled)
set $b = \{b_i\}_{i=1}^n$ of affinelly independent vectors, there is an associated geometric simplex $\Delta_b = {\rm Conv}\{b_i\}_{i=1}^n$.
Moreover, if $S\in K\subseteq 2^{[n]}$ is a simplex in an abstract simplicial complex, then the associated $b$-realization is the
geometric simplex $R_b(S) = {\rm Conv}\{b_i\}_{i\in S}$.

For example if $\delta = (\delta_1,\dots, \delta_n)$ is defined by $\delta_i = e_i-\frac{u}{n}$, where $u = e_1+\dots+ e_n$, then
\begin{equation}
      \Delta_{\delta} = {\rm Conv}\{\delta_i\}_{i=1}^n \quad \mbox{ {\rm and} }\quad  \Delta_{-\delta} = {\rm Conv}\{-\delta_i\}_{i=1}^n \, .
\end{equation}

\medskip
If $T\subseteq [n]$ then $\overline{T}$ is the corresponding subset of $[\bar{n}] = \{\bar{1}, \bar{2},\dots, \bar{n}\}$.
The `tautological geometric realization' of the abstract simplicial complex $Bier(K) = K\ast_\Delta K^\circ \subset 2^{[n]}\ast 2^{[\bar{n}]}$
is the geometric simplicial complex
\begin{equation}\label{eq:tautological}
\mathcal{R}_{\pm e}(Bier(K)) = \{R_e(S)\ast R_{-e}(T) \mid (S, T)\in K\ast_\Delta K^\circ\}
\end{equation}
where $R_e(S)\ast R_{-e}(T) = {\rm Conv}(R_e(S)\cup R_{-e}(T)) \subset \partial\lozenge_{[n]}$ is the geometric join of simplices.
Similarly, we define the `canonical geometric realization' $\mathcal{R}_{\pm\delta}(Bier(K))$ by replacing $e$ and $-e$ in
(\ref{eq:tautological}) respectively by $\delta$ and $-\delta$. By construction
\begin{equation}
\mathcal{R}_{\pm\delta}(Bier(K)) = \pi(\mathcal{R}_{\pm e}(Bier(K))) \subset H_0
\end{equation}
 where $\pi : \mathbb{R}^n \rightarrow H_0 := \{x\in \mathbb{R}^n \mid \langle u,x \rangle = 0\}$ is the orthogonal projection.

 It remains to be shown that $\mathcal{R}_{\pm\delta}(Bier(K))$ is indeed a geometric realization of the abstract simplicial complex
 $Bier(K)$ and that it is precisely the desired starshaped realization.

\begin{theo}\label{thm:Bier-starshaped}
Let $K\subsetneq 2^{[n]}$ be a simplicial complex and let $Bier(K)$ be the associated Bier sphere. Then $\mathcal{R}_{\pm\delta}(Bier(K))$
is a geometric realization of the abstract simplicial complex $Bier(K)$ which is starshaped as a subset of $H_0 := \{x\in \mathbb{R}^n \mid \langle u,x \rangle = 0\}$.
\end{theo}

\medskip\noindent
{\bf Proof:} Let ${\rm cone}(C) = \cup_{\lambda \geq 0}~\lambda C$ be the convex cone with the apex at the origin generated by a convex set $C\subset H_0$.
The theorem will follow from the observation that the collection of convex cones
\begin{equation}\label{eq:cone}
{\rm Cone}_{\pm\delta}(K) = \{{\rm Cone}(R_\delta(S)\ast R_{-\delta}(T)) \mid (S,T)\in K\ast_\Delta K^\circ\}
\end{equation}
is a complete simplicial fan in $H_0$.  Recall \cite[Chapter~III]{Ewald} that a family $\Sigma$ of simplicial cones (in $V$) with apex $0$ is
a complete simplicial fan if $\Sigma$ is a covering of $V$ and for each two cones $C_1, C_2\in \Sigma$ the intersection $C_1\cap C_2$ is
their common face which is also an element in $\Sigma$.

\medskip
We establish that (\ref{eq:cone}) is a complete fan by showing that the associated geometric ``shore subdivision'' ${\rm Shore}_{\pm\delta}(K)$ (see \cite[Section~4.3]{long-2}),
obtained by the shore subdivision of each  ${\rm cone}(R_\delta(S)\ast R_{-\delta}(T))\in {\rm Cone}_{\pm\delta}(K)$, coincides with the fan $\Sigma$ generated by the barycentric
subdivision of the boundary of the simplex $\Delta_\delta$.

\medskip
In the sequel we denote by $\nu(\Delta)$ the barycenter of a geometric simplex $\Delta$. If $\Delta = R_b(S)$ we also write $\nu_b(S):= \nu(R_b(S))$.
By definition each cone $C\in {\rm Shore}_{\pm\delta}(K)$, subdividing ${\rm cone}(R_\delta(S)\ast R_{-\delta}(T))$,  is positively spanned by the vectors
\begin{equation}\label{eq:shore-1}
\nu_\delta(S_1),\dots,  \nu_\delta(S_p), \nu_{-\delta}(T_q), \nu_{-\delta}(T_{q-1}),\dots, \nu_{-\delta}(T_1)
\end{equation}
where
\begin{equation}\label{eq:shore-chain-1}
S_1\subset S_2 \subset \dots \subset S_p \subseteq S \quad \mbox{ {\rm and} } \quad  T_1\subset T_2 \subset \dots \subset T_q \subseteq T \, .
\end{equation}
On the other hand the cone spanned by (\ref{eq:shore-1}) coincides with the cone positively spanned by the vectors
\begin{equation}\label{eq:shore-2}
\nu_\delta(S_1),\dots,  \nu_\delta(S_p), \nu_{\delta}(T_q^c), \nu_{\delta}(T_{q-1}^c),\dots, \nu_{\delta}(T_1^c)
\end{equation}
(where $T_j^c = 2^{[n]}\setminus T_j$). In light of the fact that $(S,T)\in K\ast_\Delta K^\circ$, the condition (\ref{eq:shore-chain-1}) is equivalent to
the condition
\begin{equation}\label{eq:shore-chain-2}
S_1\subset S_2 \subset \dots \subset S_p \subseteq S \subset T^c \subseteq T_q^c \subset T_{q-1}^c \subset \dots \subset T_1
\end{equation}
This is precisely the condition that the positive span of (\ref{eq:shore-chain-2}) is a cone in $\Sigma$. \hfill $\square$

\bigskip
The reader familiar with \cite{long-1} (see also \cite{long-2}) will agree that the proof of Theorem~\ref{thm:Bier-starshaped} can be concisely described as a
geometrization of the short and elegant proof of Mark de Longueville that $Bier(K)$ triangulates a sphere. Note however that the very existence of a canonical starshaped
realization $\mathcal{R}_{\pm\delta}(Bier(K))$ of $Bier(K)$ is interesting in itself and have some interesting consequences. For example it allows to compare Bier spheres by the volume of the associated starshaped body
\begin{equation}
     {\rm Star}(K) = \{\lambda x\in H_0 \mid x\in \mathcal{R}_{\pm\delta}(Bier(K)) \mbox{ {\rm and} } 0\leq \lambda \leq 1  \} \, .
\end{equation}
Moreover, it allows us to give a geometric interpretation of the classification of autodual simplicial complexes described in \cite{tim}.

\section{Kantorovich-Rubinstein polytopes}
\label{sec:K-R}

Let $(X,\rho)$, $\vert X\vert =n$, be a finite metric space and let $V(X) := \mathbb{R}^X \cong \mathbb{R}^n$ be the associated vector space
of real valued functions (weight distributions, signed measures) on $X$. Let $V_0(X) := \{\mu\in V(X) \mid \mu(X)=0\}$ be the
vector subspace of measures with total mass equal to zero, and $\Delta_X := \{\mu\in V(X)\mid \mu(X) = 1  \mbox{ {\rm and} } (\forall x\in X)\, \mu(\{x\})\geqslant 0\}$
the simplex of probability measures.

\medskip
Let $\mathcal{T}_\rho(\mu, \nu)$ be the cost of the optimal transportation of measure
$\mu$ to measure $\nu$, where the cost of transporting the unit mass from $x$ to $y$ is $\rho(x,y)$. Then, as shown in \cite{v13, Vil}, there exists a norm
$\|\cdot \|_{KR}$ on $V_0(X)$ (called the Kantorovich-Rubinstein norm), such that,
\[
             \mathcal{T}_\rho(\mu, \nu) = \|\mu - \nu \|_{KR},
\]
for each pair of probability measures $\mu, \nu\in \Delta_X$.

\begin{defin}\label{def:Kantor}
The Kantorovich-Rubinstein polytope $KR(\rho)$, associated to a finite metric space $(X,\rho)$, is the unit ball of the $KR$-norm in $V_0(X)$,
\begin{equation}\label{eqn:KR-poly}
  KR(\rho) = \{x\in V_0(X) \mid \|x \|_{KR} \leqslant 1\}.
\end{equation}
\end{defin}

The following explicit description for $KR(\rho)$ can be deduced from the Kantorovich-Rubinstein theorem (Theorem~1.14 in \cite{Vil}),
\begin{equation}\label{eqn:KR-conv}
  KR(\rho) =  {\rm Conv}\left\{  \frac{e_x - e_y}{\rho(x,y)} \mid x,y\in X \right\},
\end{equation}
where $\{e_x\}_{x\in X}$ is the canonical basis in $\mathbb{R}^X$. More information about KR-polytopes can be found in \cite{dh, gp, jjz}.

\subsection{Metrics induced by weighted graphs}

Let $\Gamma$ be a simple graph on the set of vertices $V(\Gamma) = [n]$ with the set of edges $E(\Gamma) \subset 2^{[n]}$.
We say that the graph $\Gamma$ is positively weighted if we have chosen a positive weight function $w:E(\Gamma)\rightarrow \mathbb{R}_+$.

\begin{defin}\label{def:metric_space_graph}
Let $\Gamma=\Gamma([n],E(\Gamma),w)$ be a connected graph with a positive weight function $w:E(\Gamma)\rightarrow \mathbb{R}_+$.
The associated ``geodesic metric'' $d_\Gamma$ on $[n]$ is defined by
\begin{align}\label{eq:path}
d_\Gamma(i,j)=d_{i,j}=\underset{S \in \mathcal{P}_{ij}}{\min}  \sum_{ e \in S}  w(e)  \, ,
\end{align}
where $\mathcal{P}_{ij}$ is the collection of all paths connecting vertices $i$ and $j$.
\end{defin}

\noindent
Definition~\ref{def:metric_space_graph} is meaningless if the graph is not connected so in all subsequent statements
we tacitly assume that $\Gamma$ is a connected graph.
\begin{lema}\label{lem:conv}
Let $\left([n],d_\Gamma \right)$ be the geodesic metric space induced by a positively weighted graph $\Gamma=([n],E(\Gamma),w)$.
Then
\begin{align}
KR(d_\Gamma) = \mathrm{Conv} \left( \{\pm v_{i,j} \}_{\{i,j\} \in E(\Gamma)} \right),
\end{align}
where $v_{i,j}=\frac{e_i-e_j}{d_{i,j}}$.
\end{lema}

\medskip\noindent
{\bf Proof:}
Assume that $\{k,l\} \not \in E(\Gamma)$.
Suppose that $S\in \mathcal{P}_{kl}$ is a path connecting $k$ and $l$ where the minimum in (\ref{eq:path}) is attained. By re-enumerating the vertices we may assume that
$S = (\{k,{k+1}\},\{k+1,{k+2}\},\dots,\{l-1,l\})$. Let
\begin{align*}
\alpha_i&=\frac{d_{k+i,k+i+1}}{\sum_{j=0}^{l-k-1}d_{k+j,k+j+1}}, \textrm{ for } i \in \{0,\ldots,l-k-1\}.
\end{align*}
Then $\sum_{j=0}^{l-k-1}\alpha_j=1$ and
$v_{k,l}=\alpha_0 v_{k,k+1}+\ldots+\alpha_{l-k-1}v_{l-1,l}$ which completes the proof of the lemma.  \hfill $\square$

\medskip\noindent
Let us recall that if $K={\rm Conv}(X\cup\{\pm v\})$ is such that $0\in {\rm Conv}(X)$ and $v$ is not in the vector subspace spanned by $X$, than $K$ is a suspension over ${\rm Conv}(X)$,
\begin{align*}
K = {\rm Susp}\left( {\rm Conv}(X) \right).
\end{align*}
As a direct corollary we get the following lemma.

\begin{lema}\label{lema:KR-susp}
Let $\Gamma$ be a positively weighted graph, $X$ subset of its vertices, and $x \in X$ such that $\Gamma|_{X^c \cup \{x\}}$ is connected, $\Gamma|_X$ is a non-trivial tree, and $E(\Gamma) = E(\Gamma|_{X^c \cup \{x\}}) \cup E(\Gamma|_X)$.
Then $KR(\Gamma)$ can be expressed as an iterated suspension, $$KR(\Gamma) = {\rm Susp}^{|X|-1} (KR(\Gamma|_{X^c\cup\{x\}}))\, .$$
\end{lema}

\begin{proof}
If $|X|=2$ we may assume without loss of generality that $X=\{n-1,n\}$ and that $n$ is the isolated vertex.
Since $v_{n-1,n}$ is not in the linear span of $KR(\Gamma|_{X^c\cup\{x\}})$,  there is an isomorphism $KR(\Gamma) \cong {\rm Susp}(KR(\Gamma|_{X^c\cup\{x\}}))$.
The general case of the lemma follows by induction on $|X|$.
\end{proof}

\noindent
The following proposition is an immediate consequence of Lemma~\ref{lema:KR-susp}.

\begin{prop}
Let $T = ([n], E(T), w)$ be a positively weighted tree on $[n]$. Then $KR(d_T) \cong \lozenge_{[n-1]}$. In other words, all $KR$-polytopes associated to weighted trees
on $[n]$ are combinatorially equivalent.
\end{prop}

\subsection{Proof of Theorem~\ref{thm:KR-Bier}}

Theorem~\ref{thm:KR-Bier} says, in a nutshell, that the boundary of the Kantorovich-Rubinstein polytope of the geodesic metric
$d_L$ coincides with the complex of short sets of the associated linkage $\hat{L}$. Here we obtain this result as a corollary of Theorem~\ref{thm:Bier-polytope}.

\medskip\noindent
{\bf Proof of Theorem~\ref{thm:KR-Bier}}: Let $\Gamma$ be a positively weighted cycle. More explicitly
$\Gamma = (V(\Gamma), E(\Gamma), w)$ is a positively weighted graph where
\begin{align*}
V(\Gamma) = [n], E(\Gamma) = \{ \{1,2\}, \{2,3\}, \dots, \{n-1,n\}, \{n,1\} \}
\end{align*}
and $w(\{i, i+1\}) = l_i$ for each $i\in [n]$. (Here and to the end of this section we use the convention that $n+1:= 1$.)

\medskip
By Lemma~\ref{lem:conv},
\begin{equation}
KR(d_L) = {\rm Conv}\left\{\pm \frac{u_i}{l_i}\right\}_{i=1}^n = {\rm Conv}\{\pm y_i \}_{i=1}^n
\end{equation}
where  $y_i = \frac{u_i}{l_i} = \frac{e_{i+1}-e_i}{l_i}$.  By comparison with (\ref{eqn:conv-poly}) we observe that $\alpha = \beta =1$ and
\begin{equation}
KR(d_\Gamma) = Q_1 = {\rm Conv}(\Delta \cup \nabla) \, .
\end{equation}
Finally by (\ref{eqn:final}) we observe that
\begin{equation}
Q_1 = {\rm Short}(L) := \{I\subseteq [n] \mid \mu_L(I) < 1/2\}
\end{equation}
and the result follows as a consequence of Theorem~\ref{thm:Bier-polytope}. \hfill $\square$

\medskip\noindent
It is known, see \cite{far}, that moduli spaces $M_L$ and $M_{L^\sigma}$ of two linkages $L = (l_1,\dots, l_n)$ and $L^\sigma = (l_{\sigma(1)},\ldots,l_{\sigma(n)})$,
where $\sigma\in \Sigma_n$ is a permutation, are homeomorphic. Since ${\rm Short}(L) \cong {\rm Short}(L^\sigma)$, it follows from Theorem~\ref{thm:KR-Bier}  that
the polytopes $KR(d_L)$ and $KR(d_{L^\sigma})$ are  combinatorially isomorphic. Here we give a direct proof by constructing an explicit affine isomorphism.

\begin{prop}
Let $L=(l_1,\ldots,l_n)$ and $L^\sigma=(l_{\sigma(1)},\ldots,l_{\sigma(n)})$ where $\sigma \in \Sigma_n$ is a permutation.
Then $KR(L) \cong KR(L^\sigma)$.
\end{prop}
\begin{proof}
It is sufficient to prove the statement when  $\sigma$ is a cycle or transposition.
If $\sigma$ is a cycle, the statement is equivalent to relabeling the basis vectors.
Let $\sigma$ be a transposition.
Without loss of generality, let $\sigma =(1,2)$.
Notice that there exists a hyperplane containing $e_4-e_3,\ldots,e_{n}-e_{n-1},e_1-e_n$ which bisects the angle between $e_2-e_1$ and $e_3-e_2$.
The reflection with respect to that hyperplane sends $KR(L)$ to $KR(L^\sigma)$.
\end{proof}

\medskip
\subsection*{Acknowledgements} It is our pleasure to acknowledge valuable remarks and useful suggestions by B. Sturmfels, G.M. Ziegler, S. Vre\' cica, D. Joji\' c, V. Gruji\' c,
  members of Belgrade CGTA-seminar, and the anonymous referees.  The authors were supported by the Grants 174020 and 174034
of the Ministry of Education, Science and Technological Development of
Serbia.

\end{document}